\RequirePackage{fix-cm}
\documentclass[smallextended]{svjour3}       
\smartqed  
\usepackage[T1]{fontenc}
\usepackage[utf8]{inputenc}
\usepackage[english]{babel}
\usepackage{graphicx}
\usepackage{amsmath,amssymb,amsfonts,stmaryrd,mathtools}
\usepackage[activate={true,nocompatibility},spacing,kerning]{microtype}

\journalname{Journal}

\makeatletter
\renewcommand{\leq}{\leqslant}
\renewcommand{\geq}{\geqslant}
\DeclareMathOperator{\Ai}{{\operator@font Ai}}
\DeclareMathOperator{\cl}{{\operator@font cl}}
\let \Re \relax
\DeclareMathOperator{\Re}{Re}
\let \Im \relax
\DeclareMathOperator{\Im}{Im}

\let \limsup \relax
\DeclareMathOperator*{\limsup}{\smash[b]{\operator@font \overline{\lim}}}
\makeatother
\newcommand{\C}{{\mathbb  C}}

\newcommand{\R}{{\mathbb  R}}
\newcommand{\Z}{{\mathbb  Z}}
\newcommand{\N}{{\mathbb  N}}

\begin{document}

\title{A Jentzsch-Theorem for Kapteyn, Neumann, and General Dirichlet Series
}

\titlerunning{A Jentzsch-Theorem for Kapteyn, Neumann, and General Dirichlet Series}        

\author{Folkmar Bornemann}


\institute{Folkmar Bornemann \at
              Department of Mathematics\\ 
              Technical University of Munich \\
              \email{bornemann@tum.de}}

\date{}

\maketitle

\begin{abstract}
Comparing phase plots of truncated series solutions of Kepler's equation by Lagrange's power series with those by Bessel's Kapteyn series strongly suggest that a Jentzsch-type theorem holds true not only for the former but also for the latter series: each point of the boundary of the domain of convergence in the complex plane is a cluster point of zeros of sections of the series. We prove this result by studying properties of the growth function of a sequence of entire functions. For series, this growth function is computable in terms of the convergence abscissa of an associated general Dirichlet series. The proof then extends, besides including Jentzsch's classical result for power series, to general Dirichlet series, to Kapteyn, and to Neumann series of Bessel functions. Moreover, sections of Kapteyn and Neumann series generally exhibit zeros close to the real axis which can be explained, including their asymptotic linear density, by the theory of the distribution of zeros of entire functions. 
\end{abstract}

\section{Introduction}\label{sec:intro}

The story of this paper starts, as quite a few in the history of mathematics \cite{MR1268639}, with Kepler's equation of 1609, namely
\[
M = E -\epsilon \cdot \sin E,
\]
where $M$ is the mean anomaly and $E$ the eccentric anomaly of a celestial body on an elliptic orbit of eccentricity $\epsilon $, such as a planet or a recurrent comet in the solar system. By solving for $E$, given $M$ and $\epsilon $, one would find the position of the body given the mean time of observation. Lagrange, by using his method of series inversion, obtained in 1771 a power series expansion in $\epsilon $, namely
\begin{equation}\label{eq:lagrange}
E-M = \sum_{n=1}^\infty \frac{\epsilon^n}{2^{n-1}n!}\sum_{0\leq k \leq n/2} (-1)^k\binom{n}{k}(n-2k)^{n-1} \sin((n-2k)M),
\end{equation}
which generally converges for all real $M$ only if $0\leq \epsilon < 0.66274\cdots$,\footnote{A singularity analysis shows that the convergence radius $\rho$ for fixed $M$ is the smallest value of $|z_*|$ such that
$(z_*,w_*)$ satisfy the equations $z_* \cos w_* = 1$, $w_* = M + z_*\sin w_*$. This radius becomes minimal for $M=\pi/2$, which then gives Laplace's limit value, see \cite[§231]{MR0060009}. For other $M$ it can be larger, e.g., for $M=1/5$ as in Fig.~\ref{fig:kepler} we get $\rho=0.84889\cdots$.}
a restriction obtained by Laplace in 1799 and justified by Cauchy in 1829. By expanding into a Fourier series with respect to $M$, Bessel found in 1818 that
\begin{equation}\label{eq:bessel}
E-M = 2 \sum_{n=1}^\infty \frac{\sin(n M)}{n} J_n(n \epsilon),
\end{equation}
where $J_n(z)$ denotes the Bessel function of the first kind of order $n$. This series converges for all $0\leq \epsilon < 1$ of interest. Series of such a form with respect to $\epsilon$ are called {\em Kapteyn series} due to the fact that they were first systematically studied, as function of a complex variable $\epsilon$, by Kapteyn in 1893, cf. \cite[p.~551]{MR0010746}. 

\begin{figure}[tbp]
\includegraphics[width=0.475\textwidth]{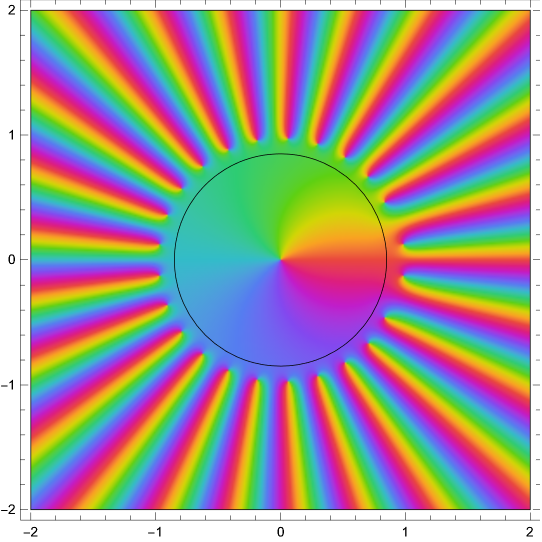}\hfil
\includegraphics[width=0.475\textwidth]{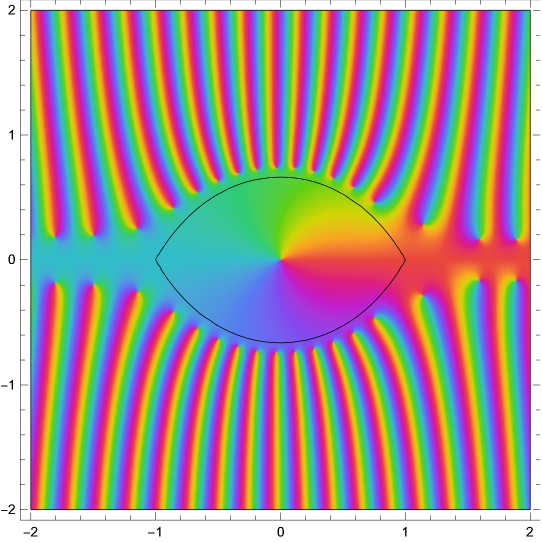}
\caption{Phase plots for complex eccentricity $\epsilon $ of the series solution of Kepler's equation for $M=1/5$ truncated at $n=25$; left: power series \eqref{eq:lagrange} with $\rho=0.84889\cdots$, right: Kapteyn series \eqref{eq:bessel} with $\rho=1$; in both cases the boundary of the domain of convergence is shown as the solid black line, cf. Fig.~\ref{fig:Drho} for the domains of convergence of Kapteyn series in general}
\label{fig:kepler}
\end{figure}

Now, during a lecture on complex analysis given by the present author the question came up whether one could see the different convergence behavior in visualizing the truncated series (called from now on {\em sections}) as functions of a complex variable $\epsilon$, e.g., by using phase plots \cite{MR3024399}. For instance, taking the mean anomaly $M=1/5$ and truncating both series at $n=25$ yields the phase plots shown in Fig.~\ref{fig:kepler}.

Strikingly, and more densely so for increasingly larger indices of truncation, zeros of the sections appear to cluster at the boundary of the domain of convergence in {\em both} cases. Whereas this is a classical theorem for power series obtained by the mathematician-poet Jentzsch in his 1914 thesis \cite{MR1555151}, the present author was not able to find any mention of such a phenomenon for Kapteyn series in the literature.\footnote{The manifold generalizations of Jentzsch's theorem are always addressing families of (exponential) polynomials, but no other families of entire transcendental functions.}  Fig.~\ref{fig:kapteyn} shows, for some random coefficients, that this phenomenon persists for different domains of convergence of Kapteyn series, even unbounded ones. 

In this paper we will prove a Jentzsch-type theorem, namely Theorem~\ref{thm:jentzsch_series} in Sect.~\ref{sec:series}, for general Kapteyn series and we will also explain, in Sect.~\ref{sec:entire}, the appearance of infinitely many zeros of sections close to the real axis---even if the domains of convergence are bounded, see the right panel in Fig.~\ref{fig:kepler} and the left one in Fig.~\ref{fig:kapteyn}. Moreover, as Fig.~\ref{fig:neumann} illustrates, similar phenomena can be observed for Neumann series of Bessel functions \cite[Chap.~XVI]{MR0010746}, that is, series of the form
\begin{equation}\label{eq:neumann}
\sum_{n=0}^\infty a_n J_n(z).
\end{equation}
Also for these series a Jentzsch-type theorem appears to be new and is therefore included in the discussion of Sect.~\ref{sec:series} and Sect.~\ref{sec:entire}.

\begin{figure}[tbp]
\includegraphics[width=0.475\textwidth]{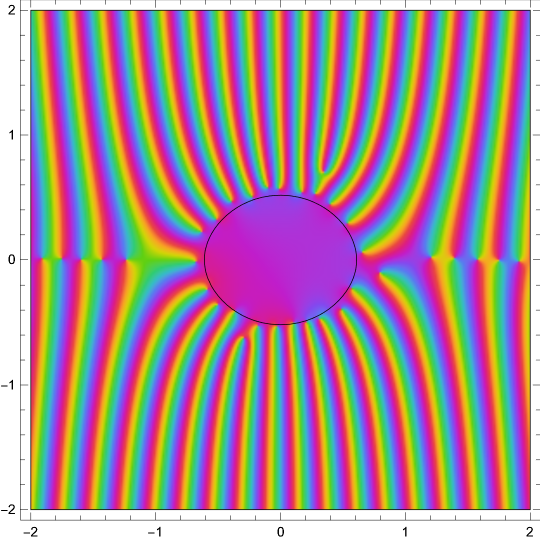}\hfil
\includegraphics[width=0.475\textwidth]{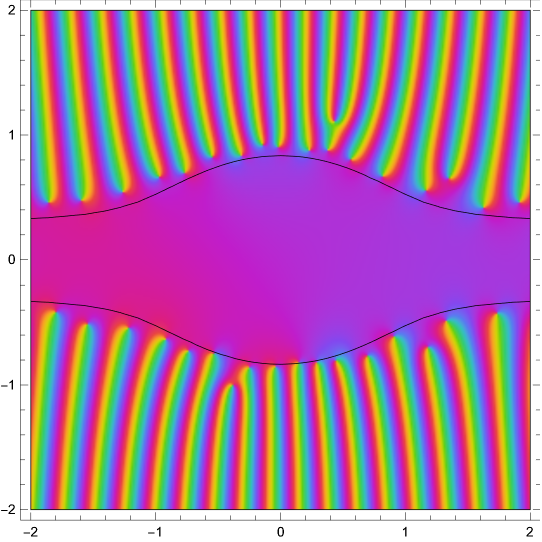}
\caption{Kapteyn series truncated at $n=25$; left: convergence level $\rho = 3/4$, right: $\rho=4/3$ }
\label{fig:kapteyn}
\end{figure}

The key to our proof of Jentzsch-type theorems for Kapteyn and Neumann series will be the following observation by Ganelius \cite[p. 33]{MR62826}:
\begin{quote}
From the original proof of {\sc Jentzsch's} theorem on the clustering of the zeros of the polynomial sections of a power series it is seen that the essential fact is not that the circle of convergence is the boundary of the domain of uniform convergence. It is the behaviour of the [growth] function
\[
\mu(z) = \limsup_{n\to \infty} |f_n(z)|^{1/n}
\]
which is important. [Here, $f_n(z)$ denotes the section of the series obtained from truncation at $n$.]
\end{quote}
This growth function is also used by Luh in \cite{MR2419473} in proving a generalization of Jentzsch's theorem to a wealth of other values than zeros. Now, if we think of $f_n(z)$ for a fixed $z$ as sections of a general series and replace the exponents $1/n$ by $1/\lambda_n$ with monotonically increasing $\lambda_n\to\infty$, we are led to consider 
\[
\log \limsup_{n\to\infty} \left|\sum_{k=0}^n a_k\right|^{1/\lambda_n} = \limsup_{n\to\infty}\lambda_n^{-1}\log\left|\sum_{k=0}^n a_k\right|,
\]
a quantity which, by a classical result of Cahen from 1894, equals $\max(0,\sigma_c)$, regardless whether the series converges or not, unless $\sum_{n=0}^\infty a_n =0$. Here $\sigma_c$ denotes the convergence abscissa of the associated general Dirichlet series
\begin{equation}\label{eq:genDir}
\sum_{n=0}^\infty a_n e^{-\lambda_n z}.
\end{equation}
It is therefore natural to include general Dirichlet series into our study of Jentzsch-type theorems. Interestingly, it was already noted by Jentzsch \cite[p.~236]{MR1555151}, and attributed by him to Knopp, that his theorem extends to sections of ordinary Dirichlet series,\footnote{Refinements of this result for the Riemann zeta function and its connections to Riemann's hypothesis have been studied by Turán \cite{MR27305}: he proved, among other things, that if the sections $f_n(z)$ do not vanish in the half-plane $\Re z>1$ for $n> n_0$, then Riemann's hypothesis is true.} that is, to the particular choice $\lambda_n = \log n$.

\begin{figure}[tbp]
\includegraphics[width=0.9125\textwidth]{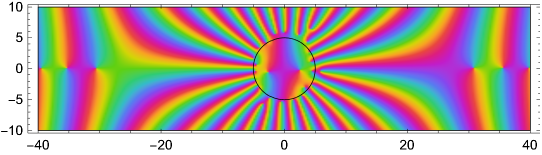}\\
\includegraphics[width=0.9125\textwidth]{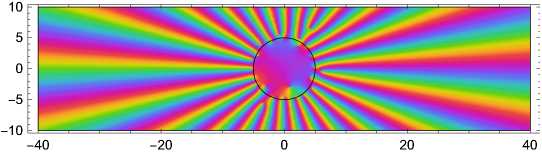}
\caption{Two series with convergence radius $\rho=5$ truncated at $n=25$; top: Neumann series, bottom: its associated power series, cf. \cite[§16.2]{MR0010746} }
\label{fig:neumann}
\end{figure}

{\em Outline of the paper.} In Sect.~\ref{sec:growth} we will introduce the growth function $\mu$ of a sequence of entire functions and derive a proto-Jentzsch-type theorem, namely Theorem~\ref{thm:proto_Jentzsch}, to hold for $\partial D \cap \cl\{\mu>1\}$,\footnote{We write $\cl A$ for the closure of a set $A\subset\C$ and concisely $\{\mu > 1\} := \{z:\mu(z)>1\}$, etc.} where $D$ is a domain of compact convergence of the sequence. We will show that this set is already all of $\partial D$ for certain parametrized domains of convergence. In Sect.~\ref{sec:series} we will apply these results to sections of the four types of series studied in this paper: general Dirichlet series, power series, Neumann series, and Kapteyn series. Finally, in Sect.~\ref{sec:entire} we will use concepts from the theory of distribution of zeros of entire functions, namely functions of completely regular growth, to explain the infinitude of zeros of sections of Kapteyn and Neumann series close to the real axis---including a prediction of their asymptotic linear density.

\vspace*{-0.4mm}

\section{Growth functions and Jentzsch-type theorems}\label{sec:growth}

In this section we will show that zeros of a sequence of entire functions $(f_n)$ cluster at those boundary points of an open set of convergence for which the sequence exhibits sufficient growth nearby. 

\begin{definition}\label{def:h}
The sequence $(f_n)$ allows an {\em admissible sequence of 
 growth exponents} $\lambda_1 < \lambda_2 < \ldots < \lambda_n \to \infty$ if for all $K\Subset \C$ (denoting $K\subset\C$ compact)
\[
\limsup_{n\to\infty} \max_{z\in K} |f_n(z)|^{1/\lambda_n} \leq c(K) < \infty.
\]
In particular, this admissibility implies that the corresponding {\em growth function}
\[
\mu(z):= \limsup_{n\to\infty} |f_n(z)|^{1/\lambda_n} 
\]
of $(f_n)$ is well defined. 
\end{definition}

We observe that $\mu(z)\leq 1$ if $f_n(z)$ converges, with equality if the limit is non-zero. Thus any open set $D$ where $(f_n)$ converges pointwise must satisfy
\begin{equation}\label{eq:cond_dom_conv}
D \cap \cl \{\mu>1\} =\emptyset.
\end{equation}

\begin{theorem}\label{thm:proto_Jentzsch} Let $(f_n)$ be a sequence of entire functions with an admissible growth function~$\mu$. If $(f_n)$ converges compactly to a non-constant holomorphic function $f$ on some open set $D$, then all points of the set
\[
\partial D \cap \cl \{\mu>1\}
\]
are cluster points of zeros of the $f_n$. 
\end{theorem}

\begin{proof}
Let us assume to the contrary that there is an $z_0 \in \partial D \cap \cl \{\mu>1\}$, some $n_0>0$ and an open disk $U$ centered at $z_0$ such that $f_n(z)\neq 0$ for all $z\in U$ and $n\geq n_0$. Then there is some $w \in U$ with $\mu(w)>1$, so that, after extracting a subsequence $n'$,
\begin{equation}\label{eq:limitw}
\lim_{n'\to\infty}|f_{n'}(w)|^{1/\lambda_{n'}} = \mu(w) > 1.
\end{equation}
Since the $f_{n'}$ are nowhere zero on the simply connected $U$, there is a choice of a univalent branch
\[ 
g_{n'}:=f_{n'}^{1/\lambda_n} 
\]
that is holomorphic in $U$. By admissibility of the growth exponents $\lambda_n$, there is some $n_0'>0$ such that with $K=\cl(U)$
\[
\max_{z\in U} |g_{n'}(z)| =\max_{z\in U} |f_{n'}(z)|^{1/\lambda_{n'}} \leq c(K) +1  \qquad (n'\geq n_0').
\]
According to Montel's theorem we can extract another subsequence $n''$ such that
$g_{n''}$ converges compactly to some holomorphic function $g$ in $U$.

If we restrict this limit to $D$, where $f_n\to f$ compactly, we get
\[
|g(z)| = \lim_{n''\to \infty} |g_{n''}(z)| =  \lim_{n''\to \infty}|f_{n''}(z)|^{1/\lambda_{n''}} = 1 \qquad (z\in U\cap D \neq \emptyset)
\]
since, according to Hurwitz's theorem, $f$ is nowhere zero in $U\cap D$. Thus, by the local mapping principle, $g$ must be locally constant with $|g|=1$ in $U\cap D$ and hence, by the identity theorem, constant on all of $U$. This, however, yields 
\[
\lim_{n''\to \infty}|f_{n''}(w)|^{1/\lambda_{n''}} = \lim_{n''\to \infty} |g_{n''}(w)| = |g(w)| = 1,
\]
which contradicts the choice of $w\in U$ we started with in \eqref{eq:limitw}. \qed
\end{proof}

To parametrize open sets of convergence we consider sublevel sets.

\begin{definition} A continuous function $\omega : \C \to [0,\infty)$ is called a {\em proper height function} if 
the open sublevel sets $D_r:= \{\omega < r\}$ are nonempty for all $r >0$ and if there are no local maxima of positive height, that is, equivalently, if
\[
 \cl \{\omega > r\} = \{\omega \geq r\} \qquad (r> 0).
\]
The {\em convergence level} of the sequence $(f_n)$ is given as
\[
\rho := \sup \{ r>0: \text{$f_n$ converges compactly on $D_r$}\},
\]
where $\rho=0$ if there is no such $D_r$ at all and $\rho=\infty$ if there is convergence in all of the $D_r$.
\end{definition}

Given a proper height function we observe the filtration property $\cl D_r \subset D_{r'}$ for $0<r<r'$, implying
that the sup defining a convergence level $0<\rho<\infty$ is indeed a max and there will be a largest level $\rho$ for which $f_n$ converges compactly on $D_\rho$. In fact, if $0<\rho<\infty$, this open set $D_\rho$ will be the maximal open set of convergence and enjoys a Jentzsch-type theorem on all of its boundary:

\begin{theorem}\label{thm:Jentzsch} Let $(f_n)$ be a sequence of entire functions that has convergence level $0< \rho <\infty$ for a proper height function $\omega$ such that the limit in $D_\rho$ is non-constant. Let there be an admissible growth function $\mu$ that satisfies
\[
\{\omega > \rho\} \subset \{ \mu > 1\} \cup \{\omega = r\},
\]
for some $r>0$. Then all points of the set $\partial D_\rho$ are cluster points of zeros of the $f_n$ and $D_\rho$ is the maximal open set in which $f_n$ converges pointwise. 
\end{theorem}

\begin{proof} The clustering of zeros at $\partial D_\rho$ follows immediately from Theorem~\ref{thm:proto_Jentzsch} by observing, for $L:=\{\omega = r\}$,
\begin{equation}\label{eq:cor1_step}
\cl \{\mu > 1\} \supset \cl\left(\{\omega > \rho\}\setminus L\right) \supset \{\omega \geq \rho\} \supset \{\omega= \rho\} \supset\partial D_\rho,
\end{equation}
so that $\partial D_\rho = \partial D_\rho \cap \cl \{\mu > 1\}$. The second inclusion is because for $r'\geq \rho$
\[
\{\omega = r'\} \subset 
\begin{cases} 
\cl(\{\omega > r'\}) = \cl\left(\{\omega > r'\}\setminus L \right) \subset \cl\left(\{\omega > \rho\}\setminus L \right) & \text{\!\!if $r' \geq r$,}\\*[2mm]
\cl(\{r > \omega > r'\}) \subset \cl\left(\{\omega > \rho\}\setminus L\right)& \text{\!\!otherwise},
\end{cases}
\]
and then taking the union over all $r'\geq\rho$.
If $D$ is any open set of pointwise convergence, then by \eqref{eq:cond_dom_conv} and \eqref{eq:cor1_step}
\[
\emptyset = D \cap \cl \{\mu > 1\} \supset D \cap \{\omega \geq \rho\},
\]
which gives $D \subset \{\omega < \rho\} = D_\rho$, so that $D_\rho$ is the maximal such open set.\qed
\end{proof}

\section{A Jentzsch-type theorem for four types of series}\label{sec:series}

We will apply the Jentzsch-type result Theorem~\ref{thm:Jentzsch} to series of the form
\begin{equation}\label{eq:series}
\sum_{n=0}^\infty a_n\phi_n(z),
\end{equation}
where $a_n$ are complex coefficients and $\phi_n$ are given entire functions. The sections of the series will be denoted by
\begin{equation}\label{eq:section}
f_n(z):=\sum_{k=0}^n a_k \phi_k(z)\qquad (n\in \N_0).
\end{equation}
Specifically, we will study the following four choices: 
\begin{itemize}
\item {\em General Dirichlet series}:\; $\phi_n(z)= e^{-\lambda_n z}$ with exponents $\lambda_n\in\R$ satisfying
\[
\lambda_0 < \lambda_1 < \lambda_2 < \cdots < \lambda_n \to \infty \qquad (n\to\infty).
\]
Here we restrict ourselves to series with an abscissa $\sigma_a < \infty$ of absolute convergence. Recall that \cite[Thm.~9]{MR0185094}
\[
\sigma_a \leq  \sigma_c + \limsup_{n\to \infty} \frac{\log n}{\lambda_n},
\]
where $\sigma_c$ denotes the abscissa of convergence. Ordinary Dirichlet series, i.e.,
series of the form
\[
\sum_{n=1}^\infty a_n n^{-z}
\]
correspond to the particular choices $a_0=0$, $\lambda_n = \log n$, that is, $\phi_n(z) = n^{-z}$. For ordinary Dirichlet series we thus have $\sigma_c \leq \sigma_a \leq \sigma_c+1$.\\*[-2mm]  
\item {\em Power series}:\; $\phi_n(z)=z^n$.\\*[-2mm]
\item {\em Neumann series}:\; $\phi_n(z) = J_n(z)$, the Bessel functions of the first kind.\\*[-2mm]
\item {\em Kapteyn series}:\; $\phi_n(z) = J_n(nz)$. They depend on 
 the proper height function\footnote{By the maximum principle, $\Omega$ cannot have any local maxima in $\C\setminus S$. For $x> 1$ we get 
 \[
 \textstyle\lim_{y\to 0^+} \partial_y \Omega(x+iy) = \sqrt{x^2-1}/x>0,
 \]
 so that, by symmetry, there are no local maxima on $S\setminus\{\pm 1\}$; also at $\pm 1$ are no local maxima.}
\[
\Omega(z):= \left|\frac{z \exp\sqrt{1-z^2}}{1+\sqrt{1-z^2}}\right| \qquad (z\in \C)
\]
which is defined by the principal branch of the square root and by continuity for $z \in S:=\{x\in \R : |x| \geq 1\}$; on $S$ both signs of the square root give the same value of $\Omega$, namely $\Omega|_S\equiv1$.
The corresponding sublevel-sets $D_r$ are visualized in Fig.~\ref{fig:Drho}. Note that $\{\Omega = 1\} = \partial D_1 \cup S$.
\end{itemize} 
We can now formulate and prove the main theorem of this paper.

\begin{figure}[tbp]
\includegraphics[width=0.48125\textwidth]{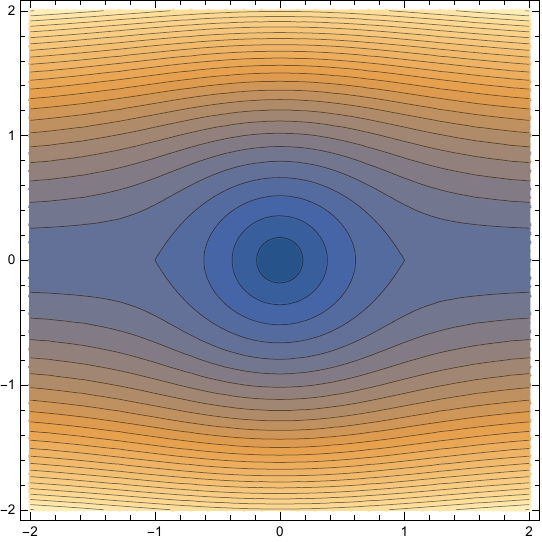}
\caption{Sublevel-sets $D_r$ of $\Omega(z)$, $r=\frac14,\frac12,\frac34,\ldots$\;}
\label{fig:Drho}
\end{figure}

\begin{theorem}\label{thm:jentzsch_series} For the four types of series at hand, Table~\ref{tab:1} lists the admissible growth exponents~$\lambda_n$, the proper height 
functions $\omega(z)$ and the convergence levels $\rho$ such that, if we assume $0<\rho<\infty$:
\begin{itemize}
\item the sections $f_n$ converge compactly in $D_\rho$ to a non-constant holomorphic $f$;\\*[-3mm]
\item the growth function satisfies
\begin{equation}\label{eq:thmgrowth}
\mu(z) = \max\left(1, \rho^{-1}\omega(z)\right)
\end{equation}
unless (a) $f_n(z)\to 0$ and $\omega(z)\leq \rho$ or (b), for the Kapteyn series, $z\in S$.
\end{itemize}
Hence, in all cases there is $\{\omega > \rho\} \subset \{\mu>1\} \cup \{\omega = 1\}$ and Theorem~\ref{thm:Jentzsch} gives that all points of the set $\partial D_\rho$ are cluster points of zeros of the $f_n$ and $D_\rho$ is the maximal open set in which $f_n$ converges pointwise. 
\end{theorem}

\begin{table}[tbp]
\caption{Growth exponents, height functions, and convergence level $\rho$ of various series}
\label{tab:1}       
\begin{tabular}{lcccc}
\hline\noalign{\smallskip}
type of series & $\phi_n(z)$ & growth expnt. $\lambda_n$ & height $\omega(z)$ & $\rho^{-1}$ \\
\noalign{\smallskip}\hline\noalign{\smallskip}
general Dirichlet$^*$ & $\exp(-\lambda_n z)$ & $\lambda_n$ & $\exp(-\Re z)$ & $e^{\sigma_c}$ \\*[2mm]
power series  & $z^n$ & $n$ & $|z|$ & $\limsup_n |a_n|^{1/n}$ \\*[2mm]
Neumann & $J_n(z)$ & $n$ & $|z|$ & $\limsup_{n} \frac{e}{2n}|a_n|^{1/n}$ \\*[2mm]
Kapteyn & $J_n(nz)$ & $n$ & $\Omega(z)$  & $\limsup_n |a_n|^{1/n}$ \\
\noalign{\smallskip}\hline
\end{tabular}\\*[1mm]
\hspace*{1mm} $^*$ with abscissae $\sigma_c$ of convergence and $\sigma_a<\infty$ of absolute convergence
\end{table}

\begin{proof} The holomorphic limit $f$ in $D_\rho$ is non-constant. In fact, since $\phi_m(z)\equiv 1$ with $m=0$ in the last three cases and, by inserting $\lambda=0$ into the sequence of growth exponents if necessary for the case of the general Dirichlet series (which would not change $\sigma_c$) so that, say, $\lambda_m=0$, uniqueness\footnote{Cf. \cite[Thm.~6]{MR0185094} and \cite[§§16.11/17.4]{MR0010746}.} of the coefficients yields for constant $f$ that $a_n=f$ for $n=m$ and $a_n = 0$ otherwise. But then $\rho=\infty$. 

In each case we will prove Eq.~\eqref{eq:thmgrowth} for $\omega(z)$, admissibility of the growth exponents $\lambda_n$, and if not already referenced, compact convergence in $D_\rho$.

\smallskip

{\em 1. General Dirichlet series.}
By the theory of general Dirichlet series \cite[Chap.~II]{MR0185094}, there is a spectral abscissa $\sigma_c \in [-\infty,\infty]$ such that $\sum_{n=0}^\infty a_n e^{-\lambda_n s}$
converges compactly to some holomorphic $f$ in the half plane $\Re s > \sigma_c$ and diverges for all $\Re s < \sigma_c$. 
Cahen's formula \cite[Thm.~7; see also the footnote there]{MR0185094} states that
\[
\limsup_{n\to \infty} \frac{\log |f_n(0)| }{\lambda_n} = \max(0,\sigma_c)
\]
unless $f_n(0)\to 0$ (which can only happen if $\sigma_c\leq 0$).
Equivalently, we get 
\begin{equation}\label{eq:cahen}
\mu(0) = \limsup_{n\to \infty} |f_n(0)|^{1/\lambda_n} = \max\left(1,e^{\sigma_c}\right) =: \max\left(1,\rho^{-1}\right)
\end{equation}
unless $f_n(0)\to 0$ and $1\leq\rho$. By applying this result for a fixed $z \in \C$ to 
\[
\sum_{n=0}^\infty \left(a_n e^{-\lambda_n z}\right) e^{-\lambda_n s} = \sum_{n=0}^\infty a_n e^{-\lambda_n (z+s)} 
\]
we get with $\omega(z):=\exp(-\Re z)$
\[
\mu(z) = \limsup_{n\to \infty} |f_n(z)|^{1/\lambda_n} = \max\left(1,e^{\sigma_c - \Re z}\right) = \max\left(1,\rho^{-1} \omega(z)\right) 
\]
unless $f_n(z)\to 0$ and $\omega(z)\leq \rho$.

Note that $-\infty<\sigma_c<\infty$ since we assume $0<\rho<\infty$. Now,
admissibility of the growth exponents $\lambda_n$ follows from $\sigma_a < \infty$, where $\sigma_a$ is the convergence abscissa of the general Dirichlet series with the coefficients $a_k$ replaced by $|a_k|$; we denote the sections of this series by $g_n(z)$. For $K\Subset \C$ we have
\[
\max_{z\in K} |f_n(z)| \leq \sum_{k: \lambda_k < 0} |a_k| e^{-\lambda_k \Re z} + \sum_{k\leq n : \lambda_k \geq 0} |a_k| e^{-\lambda_k \Re z} \leq A + |g_n(T)|,
\]
where $A$ is the maximum over $K$ of the fixed first sum of finitely many continuous terms and $T:=\min_{z\in K} \Re z$. We thus get the explicit admissibility bound
\[
\limsup_{n\to\infty} \max_{z\in K} |f_n(z)|^{1/\lambda_n} \leq  1 + \lim_{n\to\infty} |g_n(T)|^{1/\lambda_n} \leq 1 + \max\left(1,e^{\sigma_a - T}\right).
\]

{\em 2. Power series.} By the theory of power series, the series $\sum_{n=0}^\infty a_n z^n$
converges compactly to some holomorphic $f$ in the disk $|z|<\rho$ and diverges for all $|z|>\rho$ where the {\em convergence radius} $\rho$ is given by $\rho^{-1} = \limsup_{n\to\infty} |a_n|^{1/n}$. Hence, for fixed $z$, the general Dirichlet series
\begin{equation}\label{eq:power_as_Dirichlet}
\sum_{n=0}^\infty \left(a_n z^n\right) e^{-n s} = \sum_{n=0}^\infty a_n \left(z e^{-s}\right)^n
\end{equation}
converges for $|z| e^{-\Re s} < \rho$ and diverges for $|z| e^{-\Re s} > \rho$, therefore has (wrt. the variable $s$) the convergence abscissa $\sigma_c$ given by $e^{\sigma_c} = |z|/\rho$. Denoting the sections of the power series by $f_n(z)$, Cahen's formula \eqref{eq:cahen} applied to \eqref{eq:power_as_Dirichlet} gives
\begin{equation}\label{eq:powerh}
\mu(z) = \limsup_{n\to \infty} |f_n(z)|^{1/n} = \max\left(1,e^{\sigma_c}\right) = \max\left(1,\rho^{-1}|z|\right) 
\end{equation}
unless $f_n(z)\to 0$ (which can only happen if $|z|\leq \rho$).

For $K\Subset \C$ with $R := \max_{z\in K} |z|$ we get the admissibility bound
\[
\limsup_{n\to\infty} \max_{z\in K}|f_n(z)|^{1/n} \leq \limsup_{n\to\infty}\left(\sum_{k=1}^n |a_k| \, R^k\right)^{1/n} = \max\left(1,\rho^{-1} R\right)
\]
since the power series with coefficients $|a_k|$ shares the same convergence radius.

\smallskip
{\em 3. Neumann series.} We consider the Neumann series \cite[Chap.~XVI]{MR0010746}
\[
\sum_{n=0}^\infty a_n J_n(z) 
\]
and denote its sections by $f_n(z)$. Horn's large order asymptotics of the Bessel functions \cite[§8.1]{MR0010746} gives, for $z$ fixed,
\[
J_n(z) \sim \frac{1}{\sqrt{2\pi n}} \left(\frac{e z}{2n}\right)^n \qquad (n\to \infty),
\]
so that the auxiliary power series in $w$ that is given by
\[
\sum_{n=0}^\infty a_n J_n(z) w^n
\]
has the reciprocal convergence radius $\rho^{-1} |z|$ with $\rho^{-1}:=\limsup \frac{e}{2n}|a_n|^{1/n}$. Application of \eqref{eq:powerh} to the case $w=1$ yields
\[
\mu(z) = \limsup_{n\to \infty} |f_n(z)|^{1/n} = \max\left(1,\rho^{-1}|z|\right)
\]
unless $f_n(z)\to 0$, which can only happen if $1=w \leq \rho/|z|$, that is, $|z|\leq \rho$.

To prove the admissibility of the growth exponents $\lambda_n=n$ we refer to the inequality \cite[§3.31]{MR0010746}
\[
|J_n(z)| \leq \frac{|z/2|^n}{n!} e^{|\!\Im z|}  \qquad (z\in \C, n\in \N_0).
\]
This gives, for $K\Subset \C$ with $R := \max_{z\in K} |z|$, using our results for power series, the admissibility bound
\[
\limsup_{n\to\infty} \max_{z\in K}|f_n(z)|^{1/n} \leq \limsup_{n\to\infty}\left(\sum_{k=1}^n \frac{|a_k|}{2^k k!} R^k\right)^{1/n} \leq \max\left(1,\rho^{-1} R\right)
\]
since the power series with coefficients $|a_k|/(2^k k!)$ has, by Stirling's formula, once again the reciprocal convergence radius $\rho^{-1}$. The same estimates yield the compact convergence of the series in $D_\rho$: for $z \in K\Subset D_\rho$ we have $R< \rho$ so that 
\[
\sum_{n=0}^\infty |a_n| |J_n(z)| \leq \sum_{n=0}^\infty \frac{|a_n|}{2^n n!} R^n < \infty.
\]

\smallskip
{\em 4. Kapteyn series.} We consider the Kapteyn series \cite[Chap.~XVII]{MR0010746}
\begin{equation}\label{eq:kapteyn_series}
\sum_{n=0}^\infty a_n J_n(n z) 
\end{equation}
and denote its sections by $f_n(z)$. The Carlini--Meissel asymptotics\footnote{At least rigorously, this asymptotics is only stated for real $-1 < z < 1$ in the literature on Bessel functions. It easily extends, however, to complex $z\not\in S$ by using Olver's uniform large order asymptotics \cite[§11.10.4]{MR0435697} that gives, with $\delta>0$ fixed, uniformly for $|\arg z|<\pi-\delta$
\[
J_n(nz) \sim \left(\frac{4\zeta}{1-z^2}\right)^{1/4}\frac{\Ai(n^{2/3}\zeta)}{n^{1/3}} \qquad (n\to\infty).
\]
Here $\zeta$ is analytic in the complex $z$-plane cut at the negative reals such that $\zeta(1)=0$ and 
\[
w:=\frac{2}{3}\zeta^{3/2} = \log\frac{1+\sqrt{1-z^2}}{z} - \sqrt{1-z^2}\qquad (0<z<1).
\]
Observing that for $z\not\in S$ in the cut plane also $\zeta$ belongs to that cut plane, we can combine this result
with Copson's uniform large argument asymptotics of the Airy function \cite[§4.4.1]{MR0435697},
\[
\Ai(n^{2/3}\zeta) \sim \frac{e^{-n w}}{2\pi^{1/2}n^{1/6} \zeta^{1/4}} \qquad (n\to\infty,\; |\arg\zeta| < \pi-\delta),
\]
to get the assertion $J_n(nz)\sim e^{-nw}/(\sqrt{2\pi n}\,(1-z^2)^{1/4})$ as $n\to\infty$ for fixed $z\not\in S$ in the cut plane. Now, the general case $z\not\in S$ follows from the symmetry of the Bessel functions.} of Bessel functions \cite[§8.11]{MR0010746} gives, for $z\not\in S$ fixed,
\[
J_n(nz) \sim \frac{1}{\sqrt{2\pi n}\,(1-z^2)^{1/4}} \left(\frac{z \exp\sqrt{1-z^2}}{1+\sqrt{1-z^2}}\right)^n \qquad (n\to \infty),
\]
taking the principal branch of the root functions. This yields, for fixed $z\not\in S$,
\[
|J_n(nz)|^{1/n} \sim \Omega(z)\qquad (n\to\infty);
\]
hence the auxiliary power series in $w$ that is given by
\[
\sum_{n=0}^\infty a_n J_n(nz) w^n,
\]
has then the reciprocal convergence radius $\rho^{-1} \Omega(z)$ with $\rho^{-1}:=\limsup |a_n|^{1/n}$. Application of \eqref{eq:powerh} to the case $w=1$ yields, for $z \not\in S$,
\[
\mu(z) = \limsup_{n\to \infty} |f_n(z)|^{1/n} = \max\left(1,\rho^{-1}\Omega(z)\right)
\]
unless $f_n(z)\to 0$, which can only happen if $1=w \leq \rho/\Omega(z)$, that is, $\Omega(z)\leq \rho$. 


To prove the admissibility of the growth exponents $\lambda_n=n$ for $0<\rho<\infty$ we refer to Kapteyn's inequality \cite[§8.7]{MR0010746}
\begin{equation}\label{eq:kapteyn}
|J_n(n z)| \leq \Omega(z)^n  \qquad (z\in \C, n\in \N_0).
\end{equation}
This gives, for $K\Subset \C$ with $R := \max_{z\in K} \Omega(z)$, using our results for power series, the admissibility bound
\[
\limsup_{n\to\infty} \max_{z\in K}|f_n(z)|^{1/n} \leq \limsup_{n\to\infty}\left(\sum_{k=1}^n |a_k|R^k\right)^{1/n} = \max\left(1,\rho^{-1} R\right).
\]
The same estimates yield the compact convergence of the series in $D_\rho$: for $z \in K\Subset D_\rho$
we have $R< \rho$ so that 
\[
\sum_{n=0}^\infty |a_n| |J_n(n z)| \leq \sum_{n=0}^\infty |a_n| R^n < \infty.
\]
~\vspace*{-10mm}\\~\qed
\end{proof}

\bigskip

\begin{remark} If we denote by $\rho_a$ the convergence level of the series of the same type with the coefficients $a_k$ replaced by $|a_k|$ and assume $0<\rho_a < \infty$, then the proof shows in all four cases the explicit admissibility bound
\[
\limsup_{n\to\infty} \max_{z\in K} |f_n(z)|^{1/\lambda_n} \leq \eta + \max\left(1, \rho_a^{-1} R\right)
\]
for $K\Subset \C$ with $R:=\max_{z\in K} \omega(z)$. Here $\eta=1$ if the sequence of growth exponents starts with some negative entries and $\eta=0$ otherwise. 
\end{remark}

\begin{remark}
 In all four cases, stronger than the statement that $D_\rho$ is the maximal open set of pointwise convergence, the series diverges if $\omega(z)>\rho$ unless, for Kapteyn series, $z\in S$. This  is well-known for general Dirichlet and power series and was, in fact, used in the above proof; but it also follows for Neumann and Kapteyn series by looking at the auxiliary power series for $w=1$. 
 
For Kapteyn series, the set $S$ can be exceptional, indeed: e.g., 
\[
\sum_{n=1}^\infty n^{-2} J_n(nz)
\]
has convergence level $\rho=1$ and thus diverges in $\{\Omega > 1\}\setminus S$. However, Kapteyn's inequality \eqref{eq:kapteyn} yields the uniform convergence on all of $\{\Omega \leq 1\} \supset S$.
\end{remark}

\section{Zeros of sections of series of holonomic entire functions}\label{sec:entire}

As observed in Figs.~\ref{fig:kepler}--\ref{fig:kapteyn}, sections of Kapteyn series with $\rho\leq 1$ exhibit zeros that cannot be accounted for by the Jentzch-type Theorem~\ref{thm:jentzsch_series}; the same observation applies to the Neumann series in Figs.~\ref{fig:neumann}. Instead, we will predict the appearance and the density of these zeros by the theory of entire functions.

\subsection{Some results from the theory of entire functions}

We restrict ourselves to series of the form \eqref{eq:series} where the $\phi_n(z)$ are {\em holonomic} entire functions (also called {\em D-finite} entire functions), i.e., entire functions that are solutions of homogeneous linear differential equations with polynomial coefficients. This class of functions forms a ring that is closed under differentiation, indefinite integration and rescaling $\phi(\lambda z)$, see \cite[Thm.~7.2]{MR2768529}; hence, the finite sections $f_n(z)$ defined in \eqref{eq:section} are in this class, too. Considering their differential equations, the exponential function $\exp(z)$ and the Bessel functions $J_n(z)$ of integer order of the first kind are obvious examples of functions in this class.  

Holonomic entire functions $f(z)$ are entire functions of {\em completely regular growth}, cf. \cite[p.~747]{MR1401944}: i.e., they are of finite order $0 < \sigma<\infty$ and normal type such that
\begin{equation}\label{eq:indicator}
h_f(\theta) = \lim_{r\to\infty:r\not\in E_0} \frac{\log|f(r e^{i\theta})|}{r^\sigma}
\end{equation}
exists uniformly in $\theta$. Here the exceptional set $E_0$ has relative linear density zero, it is obviously related to the zeros of $f$. This {\em Phragmén--Lindelöf indicator function} $h(\theta)$ is $2\pi$-periodic, continuous and has derivatives $h_f'(\theta-0)$ from the left and $ h_f'(\theta+0) \geq h_f'(\theta-0)$ from the right that differ at most on a countable set; cf. \cite[§I.15]{MR589888}.

The relation of $h_f$ to the distribution of zeros of an entire function $f$ of completely regular growth is given by
the formula \cite[Thm.~III.3]{MR589888}
\begin{equation}\label{eq:zeros}
\lim_{r\to\infty} \frac{N_f(r;\alpha,\beta)}{r^\sigma} = \frac{1}{2\pi\sigma}\left(h_f'(\beta)-h_f'(\alpha) + \sigma^2\int_\alpha^\beta h_f(t)\,dt\right),
\end{equation}
where $N_f(r;\alpha,\beta)$ denotes the number of zeros of $f$ in the finite sector $|z|<r$ and $\alpha < \arg(z)< \beta$. Here, $\alpha$ and $\beta$ have to be points of differentiability of $h_f$.

We observe that the indicator function of a sum $f+g$ of two entire functions of completely regular growth of the same order $\sigma$ satisfies \cite[p.~52]{MR589888}
\begin{equation}\label{eq:sumindicator}
h_{f+g} (\theta) \leq \max\left(h_f(\theta),h_g(\theta)\right)
\end{equation}
with equality if $h_f(\theta)\neq h_g(\theta)$; of course $f+g$ has order $\sigma$ then, too.

\smallskip

We will need the following examples of indicator functions:
\begin{itemize}
\item {\em Exponential function.} $\exp((a- ib) z)$, with real $a,b$, has order $\sigma=1$ and the indicator function is \cite[p.~52]{MR589888}
\begin{equation}\label{eq:expind}
h(\theta) =a\cos(\theta) + b\sin(\theta),
\end{equation}
Consistent with the fact that the exponential function has {\em no} zeros, the density of zeros given by \eqref{eq:zeros} is zero for all sectors $\alpha < \arg(z)< \beta$.\\*[-2mm]

\item {\em Cosine function.} $\cos(z) = (e^{iz}+e^{-iz})/2$, using \eqref{eq:sumindicator} with the remark on equality, has
\[
h(\theta) = \max\big(-\sin(\theta),\sin(\theta)\big) = |\!\sin\theta|,
\]
since these values differ if $\sin\theta \neq 0$; the case $\sin\theta=0$ follows from continuity. Now, the counting formula \eqref{eq:zeros} 
gives, for $\delta >0$ small enough,
\begin{equation}\label{eq:coszerodensity}
\lim_{r\to\infty} \frac{N(r;\theta-\delta,\theta+\delta)}{r} = \begin{cases}
\dfrac{1}{\pi} & \quad \text{if $\theta \in \pi\Z$,}\\*[4mm]
0&  \quad \text{otherwise};
\end{cases}
\end{equation}
which is consistent with all zeros being located on the real axis at $\pi (\Z+\frac{1}{2})$.\\*[-2mm]

\item {\em Bessel function $J_n(z)$.} From Hankel's asymptotics \cite[§4.9.3]{MR0435697}, which we write here in the form
\[
J_n(z) \sim \sqrt{\frac{2}{\pi z}}\left(e^{i\left(z-\frac{n\pi}{2}-\frac\pi4\right)}\big(\tfrac{1}{2}+ O(z^{-1})\big)+e^{-i\left(z-\frac{n\pi}{2}-\frac\pi4\right)}\big(\tfrac{1}{2}+ O(z^{-1})\big)\right),
\]
uniformly as $z\to \infty$ in the sector $|\arg z|\leq \pi-\delta$ for fixed $\delta > 0$, it thus follows in the same way as for $\cos(z)$ in the bullet point above that $J_n(z)$ has order $\sigma=1$, the indicator function
\begin{equation}\label{eq:Besselind}
h(\theta) = |\!\sin\theta|,
\end{equation}
and the linear density of zeros given by \eqref{eq:coszerodensity}.
\end{itemize}

\subsection{Application to general Dirichlet, Kapteyn and Neumann series}

If the convergence level $\rho$ of one of the series in Sect.~\ref{sec:series} is finite, then infinitely many of the coefficients $a_n$ must be non-zero. So let us assume this case, let $n_0$ be the first index with $a_{n_0}\neq 0$ and let $(a_{n'})_{n'}$ be the largest subsequence of non-zero coefficients.

\begin{itemize}
\item {\em General Dirichlet series.}
Here, using \eqref{eq:expind} and \eqref{eq:sumindicator} with the remark on equality, the section $f_{n'}$ of a general Dirichlet series \eqref{eq:genDir} has order $\sigma=1$ and its indicator function is, for $\theta \in [0,2\pi]$,
\[
h_{f_{n'}}(\theta) = \max\left(-\lambda_{n_0} \cos(\theta), -\lambda_{n'}\cos(\theta)\right)
=\begin{cases}
-\lambda_{n'} \cos(\theta) &\text{if $\frac{\pi}{2} \leq \theta \leq \frac{3\pi}{2}$}\\*[2mm]
-\lambda_{n_0} \cos(\theta)&\text{otherwise}.
\end{cases}
\]
Hence, the counting formula \eqref{eq:zeros} 
gives, for $\delta >0$ small enough,
\[
\lim_{r\to\infty} \frac{N_{f_{n'}}(r;\theta-\delta,\theta+\delta)}{r} = \begin{cases}
\dfrac{\lambda_{n'}-\lambda_{n_0}}{2\pi} & \quad \text{if $\theta \in \pi(\Z+\frac{1}{2})$,}\\*[4mm]
0&  \quad \text{otherwise}.
\end{cases}
\]
\begin{remark} As an example, the sections $f_n(s)=\sum_{k=1}^n k^{-s}$ of Riemann's zeta function then satisfy
\[
\lim_{r\to\infty} \frac{N_{f_{n}}(r;\theta-\delta,\theta+\delta)}{r} = \begin{cases}
\dfrac{\log n}{2\pi} & \quad \text{if $\theta \in \pi(\Z+\frac{1}{2})$,}\\*[4mm]
0&  \quad \text{otherwise}.
\end{cases}
\]
\end{remark}
Though this result is less precise than Theorem~\ref{thm:jentzsch_series} as it does not imply the clustering of these infinitely many zeros at the axis $\Re z = \sigma_c$ as $n'\to \infty$, it gives the asymptotic density of zeros of individual sections $f_{n'}$ in arbitrary small sectors that contain the positive or negative imaginary axis.\\*[-2mm]
\item {\em Kapteyn series.} Here, using \eqref{eq:sumindicator} with the remark on equality and \eqref{eq:Besselind}, which rescales to the indicator function $h(\theta) = n |\!\sin\theta|$
for the function $J_n(nz)$ with $n\in\N$, the section $f_{n'}$ of a Kapteyn series \eqref{eq:kapteyn_series} has order $\sigma=1$ and its indicator function is
\[
h_{f_{n'}}(\theta) = n' |\!\sin\theta|.
\]
Hence, the counting formula \eqref{eq:zeros} 
gives, for $\delta >0$ small enough,
\[
\lim_{r\to\infty} \frac{N_{f_{n'}}(r;\theta-\delta,\theta+\delta)}{r} = \begin{cases}
\dfrac{n'}{\pi} & \quad \text{if $\theta \in \pi\Z$,}\\*[4mm]
0&  \quad \text{otherwise}.
\end{cases}
\]
Thus we get a rather precise knowledge about the density of zeros of individual sections $f_{n'}$ in arbitrary small sectors that contain the positive or negative real axis, namely a linear density of $n'/\pi$, which adds to the quantitative understanding of Figs.~\ref{fig:kepler}--\ref{fig:kapteyn}. In fact, by this formula we would expect about $4\times 25/\pi \approx 32$ zeros in those figures, where we actually observe a zero count of about $35$.\\*[-2mm]
\item {\em Neumann series.} Since all the non-zero terms of a Neumann series \eqref{eq:neumann} share the same order $\sigma=1$ and indicator function
$h(\theta)=|\!\sin\theta|$, \eqref{eq:sumindicator} just gives the estimate
\begin{equation}\label{eq:indneumann}
h_{f_{n'}}(\theta) \leq |\!\sin\theta|.
\end{equation}
However, generically there should be no cancellation in the limit \eqref{eq:indicator} defining the indicator function, that is, the term with the largest coefficient in absolute value gives the dominant contribution to $\log |f_{n'}(re^{i\theta})|$ and we expect equality in \eqref{eq:indneumann}. The density of zeros would then be given by \eqref{eq:coszerodensity}. This is consistent with what we observe in Fig.~\ref{fig:neumann}: infinitely many zeros with linear density close to the real axis start popping up in that example at about $|\Re (z)| \approx 30$ and in the range $30 < \Re(z)<40$ we see, indeed, about $10/\pi \approx 3$ zeros.
\end{itemize}


%
%

\begin{acknowledgements}
The author would like to thank (a) Christian Ludwig for suggesting the visualization of the two series solutions of Kepler's equation in Fig.~\ref{fig:kepler} that has started this study; and (b) Elias Wegert for commenting on a first draft of this paper.
\end{acknowledgements}

%
%

\bibliographystyle{spmpsci}      
\bibliography{paper.bib}   

\begin{thebibliography}{10}
\providecommand{\url}[1]{{#1}}
\providecommand{\urlprefix}{URL }
\expandafter\ifx\csname urlstyle\endcsname\relax
  \providecommand{\doi}[1]{DOI~\discretionary{}{}{}#1}\else
  \providecommand{\doi}{DOI~\discretionary{}{}{}\begingroup
  \urlstyle{rm}\Url}\fi

\bibitem{MR0060009}
Carath\'{e}odory, C.: Theory of functions of a complex variable. {V}ol. 1.
\newblock Chelsea Publ. Co., New York, N. Y. (1954)

\bibitem{MR1268639}
Colwell, P.: Solving {K}epler's equation over three centuries.
\newblock Willmann-Bell, Richmond, VA (1993)

\bibitem{MR62826}
Ganelius, T.: Sequences of analytic functions and their zeros.
\newblock Ark. Mat. \textbf{3}, 1--50 (1954)

\bibitem{MR0185094}
Hardy, G.H., Riesz, M.: The general theory of {D}irichlet's series.
\newblock Cambridge University Press, Cambridge (1915)

\bibitem{MR1555151}
Jentzsch, R.: Untersuchungen zur {T}heorie der {F}olgen analytischer
  {F}unktionen.
\newblock Acta Math. \textbf{41}, 219--251 (1916)

\bibitem{MR2768529}
Kauers, M., Paule, P.: The concrete tetrahedron: symbolic sums, recurrence
  equations, generating functions, asymptotic estimates.
\newblock Springer, Wien (2011)

\bibitem{MR589888}
Levin, B.J.: Distribution of zeros of entire functions, revised edn.
\newblock American Mathematical Society, Providence, R.I. (1980)

\bibitem{MR2419473}
Luh, W.: A {J}entzsch-type-theorem.
\newblock Comput. Methods Funct. Theory \textbf{8}, 199--202 (2008)

\bibitem{MR1401944}
M\"{u}ller, J.: Accelerated polynomial approximation of finite order entire
  functions by growth reduction.
\newblock Math. Comp. \textbf{66}, 743--761 (1997)

\bibitem{MR0435697}
Olver, F.W.J.: Asymptotics and special functions.
\newblock Academic Press, New York (1974)

\bibitem{MR27305}
Tur\'{a}n, P.: On some approximative {D}irichlet-polynomials in the theory of
  the zeta-function of {R}iemann.
\newblock Danske Vid. Selsk. Mat.-Fys. Medd. \textbf{24}, 36 (1948)

\bibitem{MR0010746}
Watson, G.N.: A {T}reatise on the {T}heory of {B}essel {F}unctions, 2nd edn.
\newblock Cambridge University Press, Cambridge (1944)

\bibitem{MR3024399}
Wegert, E.: Visual complex functions.
\newblock Birkh\"{a}user, Basel (2012)

\end{thebibliography}

\end{document}